\documentclass[11pt,fleqn]{article}
\usepackage{amsmath,amsfonts,amsthm,amssymb}
\newcommand{\R}{\mathbb R}
\newcommand{\C}{\mathbb C}

\renewcommand{\Re}{\mathop {\rm Re}\nolimits}
\renewcommand{\Im}{\mathop {\rm Im}\nolimits}

\newtheorem{theorem}{Theorem}
 \newtheorem{lemma}[theorem]{Lemma}

\begin{document}

\title{Numerical Ranges and Spectral Sets: the unbounded case.}

\author{Michel Crouzeix}

\maketitle

\begin{abstract}
It is known that, if $\Omega\subset \C$ is a convex set containing the numerical range of an operator $A$, then $\Omega$ is a $C_\Omega$-spectral set for $A$ with $C_\Omega\leq1{+}\sqrt2$.
We improve this estimate in unbounded cases. \end{abstract}

\paragraph{2000 Mathematical subject classifications\,:}47A25 ; 47A30

\noindent{\bf Keywords\,:}{ numerical range, spectral set}

\section{Introduction}
We consider an unbounded convex set $\Omega\neq \C$ of the complex plane and we denote by $2\alpha$ its  aperture angle, i.e. the maximal angle of sectors contained in $\Omega$.
If $A$ is a (possibly unbounded) linear operator on a complex Hilbert space $H$ with numerical range $W(A)\subset\Omega$, we show that it holds
\begin{equation}\label{eq1}
\|f(A)\|\leq K(\alpha) \sup_{z\in\Omega}|f(z)|,\quad\textrm{with }K(\alpha)=1-\frac\alpha\pi +\sqrt{2-\frac{4\alpha}\pi +\frac{\alpha^2}{\pi ^2}},
\end{equation}
for all rational functions $f$ bounded in $\Omega$, in other words, $\Omega$ is a $K(\alpha)$-spectral set for $A$. Note that, if $\alpha$ is positive, then $K(\alpha)<K(0)=1{+}\sqrt2$ and this new estimate improves the result in \cite{crpa}. For $\alpha=\pi /2$, $\Omega$ is a half-plane, $K(\alpha)=1$, we re-obtain the von Neumann inequality for the half-plane.
\medskip

\noindent\textbf{Remark.} {\it It suffices to prove the estimate \eqref{eq1} when $\Omega$ is a smooth open convex containing the closure of $W(A)$. Furthermore, using a translation and a rotation if needed, we can assume that }
\[
\{z=\rho\, e^{i\theta }: \rho > 0, |\theta |<\alpha\}\subset\Omega\subset \{z\in \C\,: \Re z>0\}.
\]

Note that, in view of the Mergelyan theorem, the estimate \eqref{eq1} is valid, not only for rational functions $f$, but also for any $f$ belonging to the algebra 
$\mathcal{A}(\Omega):=\{f\,; f$ holomorphic in $\Omega$, continuous in $\overline\Omega$ and $f$ admits a limit at $\infty\}$.
Furthermore, using a sequence of smooth convex domains
 $\Omega_n\supset \overline{W(A)}$ converging to $\overline{W(A)}$, we easily get
 \[
 \|f(A)\|\leq K(\alpha)\sup_{z\in W(A)}|f(z)|,
 \]
which shows that the numerical range $W(A)$ is a $K(\alpha)$-spectral set for the operator $A$.
\medskip

For the sake of simplicity, we work in this paper with complex-valued functions, but there would be no difficulty to generalize the following proofs to matrix-valued $f$, without changing the constants. Therefore\,\cite{paul}, the homomorphism $f\mapsto f(A)$
from the algebra ${\mathcal A}(W(A))$ into $B(H)$
is completely bounded by $K(\alpha)$.
In other words, the numerical range $W(A)$ is a complete $K(\alpha)$-spectral set for the operator $A$.

\section{Proof in the case of bounded operators}
In the following, the algebra $\mathcal{A}(\Omega)$ is provided with the norm $\|g\|_{\infty}=\max \{|g(z)|\,; z\in\overline\Omega\}$, the Hilbert space $H$ is equipped with the inner product $\langle \cdot,\cdot\rangle$ and with the associated norm $\|\cdot\|$. We use the same notation $\|\cdot\|$ for the operator norm in $B(H)$.\bigskip

We suppose
\[
\{z=\rho\, e^{i\theta }: \rho > 0, |\theta |<\alpha\}\subset\Omega\subset \{z\in \C\,: \Re z>0\}.
\]
We denote $C_\Omega$ the best constant such that
\[
\|f(A)\|\leq C_\Omega \|f\|_\infty, \quad \textrm{for all } f\in \C(z) \textrm{ with } f(\infty)=0,\textrm{  and for all }A\in B(H)) \textrm{ with }  \overline{W(A)}\subset \Omega.
\]
Notice that this estimate is also valid if $f$ is a rational function with $f(\infty)\neq 0$\,; indeed, in this case we set $f_{\varepsilon }(z)=f(z)/(1+\varepsilon z)$ with $\varepsilon >0$ and note that
\[
\|f_\epsilon (A)\|\leq C_\Omega  \|f_{\varepsilon}\|_\infty\leq C_\Omega  \|f\|_\infty,
\]
since  $f_\varepsilon(\infty)=0$ and $|1/(1{+}\varepsilon z)|\leq 1$ in $\Omega$. 
Using the fact that $\lim_{\varepsilon \to 0}f_\varepsilon (A)=f(A)$, we get $\|f(A)\|\leq C_\Omega \|f\|_\infty$.
In view of the Mergelyan theorem, this estimate is still valid for $f\in \mathcal{A}(\Omega)$.
More generally, if $f$ is holomorphic and bounded in $\Omega$, continuous in $\overline\Omega$
(but possibly not continuous at infinity), $f(A)$ is defined since $A$ is bounded\,; we still have $f_\varepsilon (A)\to f(A)$ whence $\|f(A)\|\leq C_\Omega \|f\|_\infty$.
\medskip

From now on, we work with rational functions $f$ vanishing at $\infty$.
The boundary $\partial \Omega$ is assumed smooth and counterclockwise oriented with an arclength $s$; let be $\sigma =\sigma (s)$ the corresponding point  of the boundary. We will use the notations
\begin{align*}
\mu (\sigma ,z)&:=\frac{1}{\pi}\frac{d\arg(\sigma(s){ -}z)}{ds }=\frac{1}{2\pi i}\Big(\frac{\sigma '(s)}{\sigma(s) -z }-\frac{\overline{\sigma '(s)}}{\overline{\sigma(s)} -\bar z }\Big),\\
\mu (\sigma ,A)&:=\frac{1}{2\pi i}\big(\sigma '(s)(\sigma {-}A)^{-1}+\overline{\sigma '(s)}(\overline{\sigma} {-}A^*)^{-1}\big).
\end{align*}
Recall that $\Omega$ convex implies $\mu (\sigma ,z)>0$ for $z\in \Omega$ and $\overline{W(A)}\subset \Omega$ implies $\mu (\sigma ,A)>0$, $
\forall \sigma \in\partial \Omega$.\medskip

Since $f(\infty)=0$, we may write the Cauchy formulae (for $z\in\Omega$)
\[
f(z)=\frac1{2\pi i}\int_{\partial \Omega}f(\sigma)\frac{d\sigma }{\sigma -z}, \quad f(A)=\frac1{2\pi i}\int_{\partial \Omega}f(\sigma)(\sigma I{-}A)^{-1}d\sigma .
\]
We will also introduce the Cauchy transforms of the complex conjugate of $f$
\begin{equation}\label{eq2}
g(z):=C(\overline{f},z):=\frac1{2\pi i}\int_{\partial \Omega}\overline{f(\sigma)}\frac{d\sigma }{\sigma -z}, \quad g(A):=\frac1{2\pi i}\int_{\partial \Omega}\overline{f(\sigma)}(\sigma I{-}A)^{-1}d\sigma,
\end{equation}
and finally the transforms of $f$ by the kernel $\mu(\cdot,\cdot) $
\begin{equation}\label{eq3}
S(f,z):=\int_{-\infty}^\infty f(\sigma(s))\mu (\sigma(s),z)\,ds ,\quad S=S(f,A):=\int_{-\infty}^\infty f(\sigma (s))\mu (\sigma (s),A)\,ds.
\end{equation}
From these definitions, it is clear that (for $z\in\Omega$)
\begin{equation}\label{eq3}
f(z)+\overline{g(z)}=S(f,z)\quad\text{and}\quad S^*=f(A)^*+g(A).
\end{equation}

\begin{lemma}
 Assume that the rational function $f$ satisfies $\|f\|_{\infty}\leq 1$ and $f(\infty)=0$. Then, 
 $g=C(\bar f,\cdot)$ satisfies  $g$ is holomorphic in $\Omega$, continuous in $\overline\Omega$ and $\|g\|_{\infty}\leq1-\frac{2\alpha}\pi $.
\end{lemma}
\begin{proof} Clearly, $g$ is holomorphic in $\Omega$ and it holds
\[
g(z)=\int_{-\infty}^\infty\overline{f(\sigma(s) )}\,\mu(\sigma(s),z)\,ds-\overline{f(z)},\quad\forall z\in\Omega.
\]
We extend $g$ on the boundary by setting
\[
g(\sigma_0)=\int_{-\infty}^\infty\overline{f(\sigma(s) )}\,\mu(\sigma(s),\sigma _0)\,ds,\quad\textrm{for }\sigma _0=\sigma (s_0)\in\partial\Omega.
\]
\textit{a) The restriction of $g$ onto $\partial \Omega$ is continuous.} We introduce for $m>|s_0|{+}1$
\[
g_m(\sigma_0)=\int_{-m}^m\overline{f(\sigma(s) )}\,\mu(\sigma(s),\sigma _0)\,ds.
\]
Note that we can write
\[
g_m(\sigma_0)=\int_{-m}^m\big(\overline{f(\sigma){-f(\sigma _0}}\big)\,\mu(\sigma,\sigma _0)\,ds+\overline{f(\sigma _0)}\big(\arg(\sigma (m){-}\sigma_0)-\arg(\sigma(-m){-}\sigma_0)\big)/\pi,
\]
since $\int_{-m}^m\mu(\sigma,\sigma _0)\,ds=\big(\arg(\sigma (m){-}\sigma_0)-\arg(\sigma(-m){-}\sigma_0)\big)/\pi $.
Clearly, in order to show the continuity of $g_m$ with respect to $\sigma _0$ it suffices to show the continuity of the integral part.
Note that $|f(\sigma (s)){-}f(\sigma_0)|\leq \|f'\|_{\infty}|\sigma (s){-}\sigma_0|$ and
\[
|\mu (\sigma (s),\sigma_0)|=\frac1\pi \Big|\Im\frac {\sigma '(s)}{\sigma (s){-}\sigma_0}\Big|\leq \frac{1}{\pi | \sigma (s){-}\sigma_0|}.
\]
The integrand being continuous with respect to $\sigma _0$ for $s\neq s_0$ and bounded by $\|f'\|_\infty/\pi $, the continuity of $g_m$ follows from the dominated convergence theorem.
Now, we observe that 
\[
|g(\sigma _0)-g_m(\sigma _0)|\leq 2\int_{\R\setminus(-m,m)}\mu(\sigma,\sigma _0)\,ds
=\big(2\alpha{-}\arg(\sigma (m){-}\sigma _0){+}\arg(\sigma (-m){-}\sigma _0)\big)/\pi .
\]
This shows that, for $\sigma _0$ belonging to a compact set of $\partial \Omega$, $g_m$ uniformly converges to $g$ as $m\to \infty$. Therefore $g$ is continuous on $\partial \Omega$.\bigskip

\textit{b) Continuity of $g$ in  $\overline \Omega$.}
It suffices to show that, if $z_n\to \sigma _0\in \partial \Omega$ with $z_n\in\Omega$, then $g(z_n)\to g(\sigma _0)$. For that we associate to each $z_n$ a point $\sigma _n\in \partial \Omega$ such that
\[
|z_n{-}\sigma _n|=\min\{|z_n{-}\sigma |\,: \sigma \in \partial \Omega\}.
\]
Clearly, it holds $\sigma_n\to \sigma _0$ whence, from part a), it suffices to show that $g(z_n){-}g(\sigma _n)\to0$. Using the fact that $\int_{-\infty}^\infty\mu (\sigma(s) ,z)\,ds=2{-}2\alpha/\pi $ if $z\in \Omega$ and $\int_{-\infty}^\infty\mu (\sigma(s) ,z)\,ds=1{-}2\alpha/\pi $ if $z\in \partial \Omega$, we can write
\begin{align*}
g(z_n)-g(\sigma_n)&=\int_{-\infty}^\infty\overline{f(\sigma (s))}\ \mu (\sigma (s),z_n)\,ds-\overline{f(z_n)}-g(\sigma_n)\\
&=\int_{-\infty}^\infty
\big(\overline{f(\sigma (s)){-}f(\sigma _n)}\big)\ \mu (\sigma (s),z_n)\,ds+(2{-}\tfrac{2\alpha}{\pi })\overline{f(\sigma _n)}-\overline{f(z_n)}-g(\sigma_n)\\
&=\int_{-\infty}^\infty\big(\overline{f(\sigma (s)){-}f(\sigma _n)}\big)\ (\mu (\sigma (s),z_n){-}\mu (\sigma (s),\sigma _n))\,ds+\overline{f(\sigma _n)}-\overline{f(z_n)}\\
&=\int_{-m}^m\big(\overline{f(\sigma (s)){-}f(\sigma _n)}\big)\ (\mu (\sigma (s),z_n){-}\mu (\sigma (s),\sigma _n))\,ds+\overline{f(\sigma _n)}-\overline{f(z_n)}\\
& \ +\int_{\R\setminus(-m,m)}\big(\overline{f(\sigma (s)){-}f(\sigma _n)}\big)\ \mu (\sigma (s),z_n)\,ds
.
\end{align*}
The last integral is bounded by $\big(2\alpha{-}\arg(\sigma (m){-}z_n){+}\arg(\sigma (-m){-}z_n)\big)/\pi $. Let $\varepsilon >0$ be given. We can choose $m$ such that 
$\big(2\alpha{-}\arg(\sigma (m){-}z){+}\arg(\sigma (-m){-}z)\big)/\pi\leq \varepsilon$ for all $z$ such that $|z{-}\sigma _0|\leq 1$, and we can choose $\eta>0$ such that $|\sigma _0{-}z_n|\leq \eta$ implies 
$|f(\sigma _n){-}f(z_n)|\leq \varepsilon $. Now, we observe that the previous integrand satisfies
\[
\Big|\big(\overline{f(\sigma (s)){-}f(\sigma _n)}\big)\ (\mu (\sigma (s),z_n){-}\mu (\sigma (s),\sigma _n))\Big|\leq \|f'\|_\infty\,|\sigma (s){-}\sigma _n|\frac{|z_n-\sigma _n|}{\pi |\sigma (s){-}z_n|\,|\sigma (s){-}\sigma _n|}\leq \frac{\|f'\|_\infty}{\pi }. 
\]
Hence the corresponding integral tends to zero by the dominated convergence theorem.
We deduce $\lim_{n\to\infty}|g(z_n){-}g(\sigma _n)|\leq 2\varepsilon $ for all $\varepsilon >0$, whence
$\lim_{n\to\infty}g(z_n)=g(\sigma _0)$.\bigskip

\textit{c) Proof of the bound.} From the value at a point $\sigma _0$ on the boundary
\[
g(\sigma_0)=\int_{-\infty}^\infty\overline{f(\sigma(s) )}\,\mu(\sigma(s),\sigma _0)\,ds
\]
we deduce $|g(\sigma _0)|\leq \int_{-\infty}^\infty\mu(\sigma(s),\sigma _0)\,ds=1{-}2\alpha/\pi $ on the boundary
and thus in the interior by the maximum principle.
\end{proof}

\begin{lemma}
 Assume that the rational function $f$ satisfies $\|f\|_{\infty}\leq 1$ and $f(\infty)=0$ and that the bounded operator $A$ satisfies $\overline{W(A)}\subset \Omega$. Then, $\|S(f,A)\|\leq 2{-}2\alpha/\pi $.
\end{lemma}
\begin{proof}
Note that $\int_{-\infty}^\infty\mu (\sigma(s) ,z)\,ds =2{-}\frac{2\alpha}\pi$. Using the fact that 
 $\mu (\sigma ,A)$ is positive definite and self-adjoint, we see that
 \begin{align*}
 |\langle Su,v\rangle|&=\Big|\int_{\partial \Omega}f(\sigma )\langle\mu (\sigma ,A)u,v\rangle\,ds\Big|
 \leq \int_{\partial \Omega}|\langle\mu (\sigma ,A)u,v\rangle|\,ds\\
 &\leq  \int_{\partial \Omega}\langle\mu (\sigma ,A)u,u\rangle^{1/2}\langle\mu (\sigma ,A)v,v\rangle^{1/2}\,ds\\
 &\leq  \Big(\int_{\partial \Omega}\langle\mu (\sigma ,A)u,u\rangle\,ds\Big)^{1/2}
  \Big(\int_{\partial \Omega}\langle\mu (\sigma ,A)v,v\rangle\,ds\Big)^{1/2}\\
 &\leq  \Big\langle\int_{\partial \Omega}\mu (\sigma ,A)\,ds\  u,u\Big\rangle^{1/2}
 \Big\langle\int_{\partial \Omega}\mu (\sigma ,A)\,ds\ v,v\Big\rangle^{1/2}=\big(2{-}\tfrac{2\alpha}\pi\big)\,\|u\|\,\|v\|.
 \end{align*}
\end{proof}

\begin{theorem}
$C_\Omega\leq 1-\frac\alpha\pi +\sqrt{2-\frac{4\alpha}\pi +\frac{\alpha^2}{\pi ^2}}$.
\end{theorem}
\begin{proof} We consider a rational function $f$ such that $\|f\|_{\infty}\leq 1$ and $f(\infty)=0$ and we assume that the bounded operator $A$ satisfies $\overline{W(A)}\subset \Omega$.
Following a suggestion by Felix Schwenninger~\!\footnote{Suggestion made during  a workshop on Crouzeix's Conjecture hosted by the American Institute of Mathematic}, we deduce from $S^*=S(f,A)^*=f(A)^*+g(A)$ that
 \[
( f(A)^*f(A))^2= f(A)^*f(A)\,S^*f(A)-f(A)^*f(A)g(A)f(A).
  \]
Using the fact that $f(A)g(A)f(A)=(fgf)(A)$ and $\|(fgf )(A)\|\leq C_\Omega\|fgf\|_\infty \leq C_\Omega\|g\|_\infty$,
we get
 \[
\|f(A)\|^4=\|( f(A)^*f(A))^2\|\leq(2{-}\tfrac{2\alpha}\pi )\,C_\Omega^3+(1{-}\tfrac{2\alpha}\pi )\,C_\Omega^2 .
  \]
 Whence, using that $C_\Omega$ is the supremum of $\|f(A)\|$, we get $C_\Omega^4\leq (2{-}\tfrac{2\alpha}\pi )\,C_\Omega^3+(1{-}\frac{2\alpha}\pi )C_\Omega^2$, which gives the estimate $C_\Omega\leq 1-\frac\alpha\pi +\sqrt{2-\frac{4\alpha}\pi +\frac{\alpha^2}{\pi ^2}}$.\end{proof}

\section{The case of unbounded operators}
We now consider a closed unbounded operator $A\in \mathcal L (D(A),H)$ with domain $D(A)$ densely contained in $H$. Then, its numerical range is
\[
W(A):=\{\langle Av,v\rangle\,: v\in D(A), \ \|v\|=1\}.
\]
We assume $\{z=\rho \,e^{i\theta}\,:\rho >0,\ |\theta |<\alpha\}\subset\Omega\subset\{z\in\C\,: \Re z>0\}$, Sp$(A)\subset W(A)$ and $\overline{W(A)}\subset\Omega$.
This ensures that $f(A)$ is well defined for all rational functions bounded in $\Omega$. 
With $\varepsilon >0$ we set $A_\varepsilon =A(I{+}\varepsilon A)^{-1}$\, then, clearly $A_\varepsilon \in \mathcal L(H)$. Noticing that 
\[
\langle A_{\varepsilon }v,v\rangle=\langle A(I{+}\varepsilon A)^{-1}v,(I{+}\varepsilon A)^{-1}v\rangle+\varepsilon \| A(I{+}\varepsilon A)^{-1}v\|^2\in \Omega+\varepsilon \| A(I{+}\varepsilon A)^{-1}v\|^2,
\]
we deduce that $\overline{W(A_\varepsilon )}\subset\Omega$.
Therefore, if $f$ is a rational function bounded in  $\Omega$,
\[
\|f(A_\varepsilon )\|\leq  C_\Omega\,\sup_{z\in \Omega}|f(z)|\quad\textrm{ whence }
\|f(A)\|\leq   C_\Omega\,\sup_{z\in \Omega}|f(z)|.
\]
since $f(A_\varepsilon)$  converges to $f(A)$.\medskip

\noindent\textit{Proof of the convergence of $f(A_\varepsilon)$ to $f(A)$}.

 Since $f$ is a rational function, it suffices to show this for $f(z)=(\alpha{-}z)^{-1}$, $\alpha\notin \overline{\Omega}$. A simple calculation yields
 \[
 f(A_\varepsilon )-f(A)=f(A_\varepsilon )(A-A_\varepsilon )f(A)=\varepsilon \,A_\varepsilon f(A_\varepsilon )\,Af(A).
 \]
But, it is easily seen that $\|f(A)\|=\|(\alpha I{-}A)^{-1}\|\leq \frac1{d(\alpha,\Omega)}$,
 $\|Af(A)\|=\|{-}I{+}\alpha f(A)\|\leq 1{+}\frac{|\alpha|}{d(\alpha,\Omega)}$,
and similarly
 $\|A_\varepsilon f(A_\varepsilon )\|\leq 1{+}\frac{|\alpha|}{d(\alpha,\Omega)}$.
 
Therefore $\lim_{\varepsilon \to 0}\|f(A_\varepsilon )-f(A)\|=0$.


\begin{thebibliography}{11}

\bibitem{crpa} {\sc M.\,Crouzeix, C.\,Palencia},
{\em The numerical range is a $(1{+}\sqrt2)$-spectral set}, {SIAM J. Matrix Anal. Appl.}, 38 (2017), pp.~649--655, \verb+ https://doi.org/10.1137/17M1116672+.

 \bibitem{paul} {\sc V.\,Paulsen},
{\em Completely bounded maps and operator algebras}, Cambridge Univ. Press, 2002,
\verb+ https://doi.org/10.1017/CBO9780511546631+.

\end{thebibliography}
\end{document}